\documentclass[12pt]{amsart}

\usepackage{amssymb}
\usepackage{mathabx}
\usepackage{mathrsfs}
\usepackage{kotex}
\usepackage{enumitem}

\usepackage{graphicx}

\makeatletter
\@namedef{subjclassname@2010}{%
  \textup{2010} Mathematics Subject Classification}
\makeatother

\usepackage[T1]{fontenc}
\newtheorem{theorem}{Theorem}[section]
\newtheorem{corollary}[theorem]{Corollary}
\newtheorem{lemma}[theorem]{Lemma}
\newtheorem{proposition}[theorem]{Proposition}

\theoremstyle{definition}
\newtheorem{definition}[theorem]{Definition}

\newtheorem{example}[theorem]{Example}

\newcommand{\Cal}{\mathcal}
\renewcommand{\a}{\alpha}

\newcommand{\e}{\epsilon}

\newcommand{\G}{\Gamma}

\renewcommand{\l}{\lambda}

\renewcommand{\o}{\omega}
\renewcommand{\O}{\Omega}
\newcommand{\p}{\pi}

\newcommand{\pa}{\partial}
\renewcommand{\r}{\rho}

\newcommand{\sig}{\sigma}

\renewcommand{\th}{\theta}

\newcommand{\we}{\wedge}

\numberwithin{equation}{section}

\begin{document}

%%%%% To ease editing, for IMPAN journals add:

\baselineskip=17pt

%%%%%%%%%%%%%%%%

\title[ collision avoidance and  controllabilty for $n$ bodies   ]{the collision avoidance and  the controllability for $n$ bodies in dimension one    }

\author[C.-K. Han]{Chong-Kyu Han*}
\address[C.-K. Han]{Research institute of Mathematics
\\ Seoul National University\\ 1 Gwanak-ro, Gwanak-gu\\ Seoul 08826, Republic of Korea}
\email{ckhan@snu.ac.kr}
%
%\thanks{This work was  partially supported by the National Research Foundation of the Republic of Korea with grant NRF-2017R1A2B4007119.}

\author[D. Park]{Donghoon Park$^\dagger$ }
\address[D. Park]{Department of Mathematics\\Yonsei University\\50 Yonsei-Ro, Seodaemun-gu\\ Seoul 03722, Republic of Korea}

\email{dhpark12@gmail.com}

\date{\today}

\thanks{The first author was  supported by  NRF-Republic of Korea with  grants 
 0450-20210049.}

\begin{abstract}  We present a method of  control system design  for $n$ bodies in the real line $\Bbb R^1$ and on the unit circle $ S^1$, to be collision-free and controllable. The problem reduces  to designing a control-affine system  in $\Bbb R^n$ and in $n$-torus $T^n, $ respectively,  that avoids       certain obstacles.  We prove the controllability of the system by showing that the vector fields that define the control-affine system, together with  their brackets of first order, span the whole tangent space of the state space, and then  by applying the Rashevsky-Chow theorem.  
\end{abstract}



\maketitle

\section{Introduction and the statement of the main results}\label{intro}   Let $M^1$ be a smooth  $ (C^{\infty}) $ connected manifold
of dimension $1$ without boundary, namely, either  the real line $\Bbb R^1$ or the unit circle $S^1.$
We are concerned in this paper with the design of  control systems for $n$ particles on $M^1$      
to be  collision-free and controllable, which we shall call {\it the control $n$ body problem on $M^1$. }   Let   
$$M^n := \underbrace{M^1 \times \cdots \times M^1}_{ n \text{ times}} $$   be the product.    We shall see in \S 1.1 that the problem turns out to be  constructing a control-affine system on $M^n$ with the avoidance of a certain obstacle  so that for any two points of $M^n$ off the obstacle, one is reachable from the other by an admissible control.
 Thus we construct  on  $M^n$  
 a driftless control-affine system 
 \begin{equation}\label{first affine}\dot{\bf x} = \sum_{\ell = 1}^m  \vec{f_{\ell}}({\bf x}) u_{\ell} , \quad {\bf x}=(x_1,\cdots, x_n)\in M^n, ~~m\le n, \end{equation}  
 that satisfies the following conditions:
 
 \noindent i) The motions avoid the obstacle.

 \noindent ii) The Lie algebra generated by $\vec{f_{\ell}}(x), $  $\ell=1, \cdots, m,$  span the whole tangent space $T_xM$  at every point  $x\in M^n$ off the obstacle.  In terms of $n$-dimensional  kinematics, this requirement is that  {\it the constraints on the velocity  are  completely non-holonomic}  (see   \cite{AS}).  
 
 Then  the controllability follows from the Rashevsky-Chow theorem (Theorem {\ref{Chow}}).  We shall present  our results separately;   in \S 1.1 for the case  $M^1=\Bbb R^1$ and  in \S 1.2 for the case $M^1=S^1$,   respectively.       Since \cite{BMS} was published there has been extensive study on  control-affine systems and more general non-linear control systems.    We mention some of the vast literature;  
 \cite{J2014} and \cite{KM95} for non-holonomic  constraints, \cite{ACHL, CD, DLKZ, FND, GZ1, GZ3, HV, Kh1986, Kod,  MSA, PW, RDK, SVS,  TF}  
   for collision avoidance and  obstacle avoidance,  and   \cite{BL, DKK, KodR,  L, PKR,  RK,  RK2, ZG}  for navigation and motion planning problems.

What we present in \S 2 as preliminaries are either basic calculus on manifolds or well known facts to the control theorists, and hence, experts  may well skip this section.   In \S 3 we present the proofs of our results. 
  \S 4  is an epilogue newly added to the first draft of this paper mainly to answer the questions
of the referees and possibly of other readers, as well.  The authors thank the referees for  their comments and critiques.

\subsection{Control $n$ body problem  on $\Bbb R^1$}

We consider $n, n\ge 2, $ particles in the    real line $\Bbb R^1.$   Let  
\begin{equation}\label{n-body} {\bf x}(t):= (x_1(t), \cdots, x_n(t)),  \quad x_j < x_{j+1}, ~~j=1, \cdots, n-1, \end{equation} be their position at time $t. $   Once for all in this paper we fix a 
 small positive number   $\e>0. $   
We shall  say  that   (\ref{n-body}) is separated  by $\e>0$ 
 if the distance between two adjacent particles is greater than $\e, $ namely,
 \begin{equation}\label{separated} x_j + \e < x_{j+1}, \quad \text{for each }  j=1, \cdots, n-1.\end{equation}
     We want to design a  control-affine system 
\begin{equation}\label{affine} \left[\begin{matrix}\dot x_1\\ \vdots\\  \dot x_n \end{matrix}\right] = \sum_{\ell=1}^m  \vec f_{\ell} (\underbrace{x_1, \cdots, x_n}_{\bf x}) u_{\ell}(t), \quad m\le n, ~~ \end{equation}
where $\vec f_{\ell}$ are   smooth  vector fields in $\Bbb R^n,$  expressed as  $n$-dimensional column vectors,  and $u_{\ell}(t)$  are piecewise smooth,  with the following requirements:

\noindent 1) Collision free, that is, (\ref{n-body}) is separated by $\e>0$.

\noindent 2) Controllability,  that is,  for any sequences of real numbers ${\bf p}:=(p_1, \cdots, p_n)$ and ${\bf q}:=(q_1, \cdots, q_n)$  that are separated by $\e>0, $ there exist piecewise smooth real-valued functions $u_{\ell}(t)$, $\ell = 1, \cdots, m, $ in  
  some  interval   $0\le t \le T, $ so that   (\ref{affine}) has a solution ${\bf x}(t)$ with 
$x_j(0) = p_j $ and  $   x_j(T) = q_j $ for each $j=1, \cdots, n.$ 
The terminal point ${\bf x}(T) = (x_1(T), \cdots, x_n(T))$ of the initial value problem (\ref{affine}) with the initial condition ${\bf x}(0)={\bf p}$ 
shall be denoted by $\G({\bf p}, T, {\bf u}),$ where ${\bf u}(t) = (u_1(t), \cdots, u_m(t)).$  See \S 2.1 for more details.

Let $\Bbb R^n_{\e}$ be the set of points ${\bf x}=(x_1, \cdots, x_n) \in \Bbb R^n$ that satisfy (\ref{separated}). For each $j=1, \cdots, n-1,$  let 
\begin{equation}\label{rhoj} \r_j({\bf  x}) := x_{j+1}-x_j-\e.\end{equation}  Then  
\begin{equation}\label{Me} \Bbb R^n_{\e} = \bigcap_{j=1}^{n-1} \{\r_j > 0\}.\end{equation}

The requirement 1);  system's being collision-free, is equivalent to that $\Bbb R^n_{\e}$ is invariant under the flow of the vector fields $\pm \vec f_{\ell}, \ell = 1, \cdots, m$.   Our basic observations are  the following:  A smooth vector field  $\vec f = \sum_{i=1}^n a_j({\bf x}) \pa /\pa x_i $ on  $\Bbb R^n$ is tangent to a  hypersurface  given by $\{\r=0\}, $ with $ d\r \neq 0, $     if and only if 
\begin{equation}\label{basic observ}  (\vec f \r )({\bf x)}  =  0,  ~~ \quad \text {for all {\bf x} with   }  \r({\bf x}) = 0.\end{equation}
A function $\r$  that satisfies (\ref{basic observ})  is called { \it a generalized first integral } of the vector field $\vec f$ (see \S 2.2).   
For a smooth  function $\r$ on $\Bbb R^n$ let us denote by $((\r))$ the ideal generated by $\r$, namely, the collection  of all the 
smooth functions on $\Bbb R^n$ that are divisible by $\r$.
  Then (\ref{basic observ}) is saying that  
$\vec f \r $ belongs to the ideal  $((\r)); $   see  Lemma \ref{lemma} for the details.
So, in order to design a control-affine system (\ref{affine})  we  construct vector fields $\vec f_{\ell},  $  $\ell = 1, \cdots, m , $ on $\Bbb R^n$  so that 

\noindent i) each $\vec f_{\ell}$  is tangent to the boundary of $\Bbb R^n_{\e}$ (collision-free), 

\noindent ii) the Lie algebra generated by $\vec f_1, \cdots, \vec f_{m}$  spans the whole tangent space at every point of $\Bbb R^n_{\e}$ (controllability).

 Now we state our main results:
\begin{theorem}\label{thm1} Given $\e>0$ and  a positive integer $n, n\ge 2, $  let $\r_j, $ $j=1, \cdots, n-1,$ be smooth real valued functions on $\Bbb R^n$  as in (\ref{rhoj}) and  $\Bbb R^n_{\e}$  is the connected open subset of $\Bbb R^n$  given by  (\ref{Me}).  Then  there exist smooth vector fields 
$\vec{f}_1, \cdots, \vec{f}_m, $   $m\le n,$  on $\Bbb R^n$ having the following properties:

\noindent i) The hypersurfaces $\r_j({\bf x}) = 0, $  $j=1, \cdots, n-1,$ and the submanifold  $\Bbb R^n_{\e} $  are invariant under the flow of the vector fields
$\pm \vec f_{\ell}, $   $\ell=1, \cdots, m.$

\noindent ii) The control-affine system $\dot{\bf x} = \sum_{\ell=1}^m \vec f_{\ell} ({\bf x}) u_{\ell}(t)  $   
 is controllable in $\Bbb R^n_{\e}$   (see Definition \ref {controllability defn}). 

Such  $\vec f_{\ell},   \ell =1, \cdots, m, $ are  constructible for $m$  with 
\begin{equation}\label{2mgen+1} 2m \ge n+1,\end{equation} so that $\vec f_{\ell} $ depends only on $x_{\ell - 1}, x_{\ell}, x_{m+\ell -2} $ and $x_{m+\ell -1}.$
\end{theorem}

\begin{example}\label{example}  The case $n=2:$ Let
$\r (x_1, x_2) = x_2 - x_1 -\epsilon$  and $\Bbb R^2_{\e} = \{ \r > 0 \}$  be  the open half plane.  
Then $m=2$ is the smallest integer that satisfies (\ref{2mgen+1}) and 
the vector fields  
\begin{equation*}\begin{aligned} \displaystyle \vec f_1 &:=  \frac{\pa }{\pa x_1} + \frac{\pa}{\pa x_2} \\
\vec f_2  &:= (1+\r) \frac{\pa}{\pa x_1} +  \frac{\pa}{\pa x_2}, \end{aligned} \end{equation*}  as shown in Figure 1, are linearly independent in $\Bbb R^n_{\e}$.

Moreover,   
$ \vec f_{\ell} \rho = 0 $ on the set $\{\r=0\}$,  for each $\ell=1,2,$ therefore, $\vec f_{\ell}$  is  tangent to the zero set of $\rho$, which is the boundary of $\Bbb R^2_{\e}.$
It follows that $\Bbb R^2_{\e}$ is invariant under the flow of $\vec f_{\ell}$  and the controllability of (\ref{affine}) follows from the  Rashevsky-Chow theorem (Theorem \ref{Chow}).

\end{example}

\begin{figure}[h]
\includegraphics[scale=0.309]{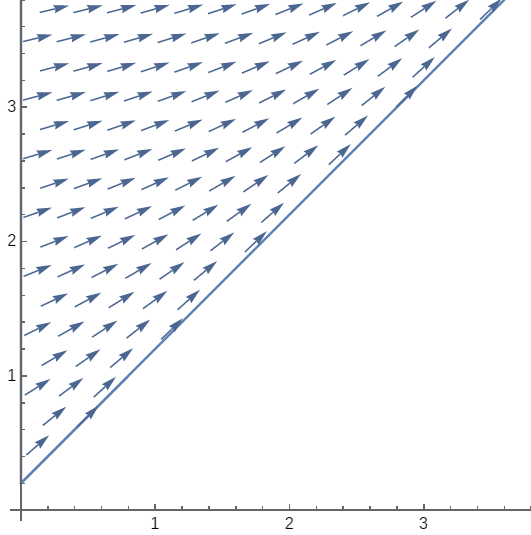}
\\$ $\\
Figure 1 $\;\;\overrightarrow{f_2}$ on the region $\mathbb{R}^2_\epsilon$
\end{figure}

\vskip 0.0pc
Recalling that  $n$ bodies in $\Bbb R^1$ were identified as a point of $\Bbb R^n, $  from  Theorem \ref{thm1}  it follows 

\begin{corollary}\label{cor1}  Given $\e>0$ and an integer $n$, $n\ge 2$, the control $n$ body problem in the real line $\Bbb R$  with separation by  $\e$ is solvable by constructing $n$-dimensional  vector fields  $\vec f_1, \cdots, \vec f_m , $   $2m \ge n+1. $
Furthermore, for each $\ell=1, \cdots, m, $   $\vec f_{\ell} $ depends only on $x_{\ell-1}, x_{\ell}, x_{m+\ell-2}$ and $x_{m+\ell-1}.$
\end{corollary}

\subsection{Control $n$ body problem  on $S^1$}     We consider now   $n, n\ge 2, $  particles moving on  $S^1.$  Let 
\begin{equation}\label{n points circle} 
e^{i x_1(t)}, \cdots, e^{ix_n(t)}  \end{equation}   be their position at time $t, $ where $x_j,$ $j=1, \cdots, n,$  is the angle measured in radians from a reference point of $S^1. $    Assuming $\e < 2\pi/n,$   we shall say that these points are separated  by $\e$ if 
\begin{equation}\label{separated on circle} x_1+\e < x_2 < x_2+\e < x_3 <\cdots < x_n + \e < x_1+ 2\p.\end{equation} 
Let ${\bf x}:=(x_1, \cdots, x_n)$ and let 
\begin{equation}\label{rjcircle}\begin{aligned} \r_j({\bf x}) &:= x_{j+1}-x_j-\e, \quad\text{for } j=1, \cdots, n-1 \\
\r_n({\bf x}) &:= x_1+2\pi - x_n-\e.\end{aligned}\end{equation}
Let  \begin{equation}\label{Repsilon} \widetilde{\Bbb R}^n_{\e} := \{{\bf x}\in \Bbb R^n ~:~ \r_j({\bf x}) > 0, ~~ j=1, \cdots, n\}, \end{equation}  that is, the set of points where (\ref{separated on circle}) holds.
Now we denote by   $2\pi \Bbb Z$ the set of all integer multiples of $2\pi $    and  consider the $n$-torus 
$$T^n:= \Bbb R^n / (2\pi \Bbb Z)^n.$$ 
Since the projection  
$$p: \Bbb R^n \rightarrow T^n $$ is a local diffeomorphism and the subset $\widetilde {\Bbb R}^n_{\e} \subset \Bbb R^n$ is open 
and connected,  we have 
     
\begin{proposition}$T_{\e}^n:= p({\widetilde R}^n_{\e})$ is an open connected subset of  $T^n$.
\end{proposition}

  In the light of  Rashevsky-Chow's theorem (Theorem \ref{Chow})    we want to construct smooth vector fields $\vec f_{\ell}, $   $\ell=1, \cdots, m,$ $m\le n, $ on $T^n, $  that satisfy the following conditions:

\noindent i) The submanifold $T^n_{\e} \subset T^n$ is invariant under the motions subject to  the control-affine system (\ref{affine}). 

\noindent ii)  $\vec f_{\ell}, $  $\ell=1, \cdots, m,$   together with their Lie-brackets of first order span the whole tangent space of $T^n_{\e}.$

A  vector field on $\Bbb R^n$ 
\begin{equation}\label{vfonRn}\vec f = \sum_{j=1}^n a_j({\bf x}) \frac{\pa}{\pa x_j} \end{equation}
shall be said to be  periodic  if each $a_j$ is a periodic function in each variable  $x_1, \cdots, x_n$.  
A periodic vector field on $\Bbb R^n$ with period $2\pi $ in each variable may be regarded as a vector field on $T^n.$  Thus we state our result as follows:

\begin{theorem} \label{thm2} Given $\e>0$ and a positive integer $n, n\ge 2, $ with  $\e< 2\pi/n, $  there exist smooth  vector fields 
$\vec f_1, \cdots, \vec f_m, $   $m\le n, $ on $\Bbb R^n = \{{\bf x}:=(x_1, \cdots, x_n)\}$, which are periodic with period $2\pi$, having the following properties: 

\noindent i) The hypersurfaces $\r_j ({\bf x}) =0, ~~j=1, \cdots, n-1, $    the hypersurface 
$\r_n({\bf x}) = 0, $  and the open submanifold $ \widetilde{\Bbb R}^n_{\e}$  are invariant under the flow of  vector fields 
$\pm \vec f_{\ell},$   where $\r_j $ and $\r_n$ are as defined by (\ref{rjcircle}).

\noindent ii) $\vec f_{\ell},  [\vec f_{\ell}, \vec f_k], $   $\ell, k = 1, \cdots, m,$ span the whole tangent space of $\widetilde{\Bbb R}^n_{\e}.$  

Such $\vec {f_{\ell}},  $ $\ell=1, \cdots, m, $ are constructible in  the following cases: 

\noindent Case 1; both $n$  and   $m$  are  odd numbers and    $2m \ge  n+1.$ 

\noindent Case 2;  $n$ is even and  $2m\ge n+2.$
\end{theorem}

\section{preliminaries }

\subsection{Orbits and controllability}

Let $M $ be a connected smooth manifold of dimension $n$ and $\mathcal{U}$ be a set of admissible controls $u: [0,\infty) \rightarrow\mathbb{R}^m$, $m\leq n$.   Let  
\begin{equation}\label{affine_control_system}
\dot{\bf x}:=\dfrac{d \bf x}{dt}= \sum_{j=1}^{m}\vec f_{\ell}({\bf x})u_{\ell}
\end{equation}   be a control-affine system on $M$, where  $\vec f_{\ell}, \ell=1, \cdots, m, $ are smooth vector fields on $M$, and ${\bf u}=(u_1,\ldots,u_m)\in\mathcal{U}$. Here the control ${\bf u}(t)$ can be chosen variously, for instance, to be piecewise continuous, measurable, smooth, and so forth.   In this paper we shall assume  that $\mathcal U$ is the set of all piecewise smooth  functions  with finitely many discontinuities.
For a point ${\bf p}\in M$ let ${\bf x}(t)$ be the solution of (\ref{affine_control_system}) subjected to the initial condition ${\bf x}(0)={\bf p}.$
The solution ${\bf x}(t)$  is  continuous,  piecewise smooth and uniquely determined by the choice of ${\bf p}$ and ${\bf u}(t)$, which we shall denote by $\Gamma({\bf p},t,{\bf u})$.
\begin{definition}
A submanifold $N$ of $M$ is said to be \emph{invariant under controls} of \eqref{affine_control_system} if ${\bf x}_{0}\in N$ implies that $\Gamma({\bf x}_0,t, {\bf u})\in N$ for all possible choices of ${\bf u}\in \mathcal U $ and for all $t\ge 0.$
\end{definition}
Now we recall some basics of non-linear control systems. For other definitions and theorems we refer the reader to our basic references~\cite{AS, BMS}. 
We consider   the set of points that are reachable from a point in some non-negative time.  

\begin{definition}\label{controllability defn} 
A point ${\bf q}\in M$ is said to be \emph{reachable from} ${\bf p}$ if 
$ q=\Gamma({\bf p},t, {\bf u}), $ for some $ {\bf u}\in \mathcal U $  and for some  $t \ge 0.$   
 The \emph{reachable set $\mathcal{R}_{{\bf p}}$ of the control system \eqref{affine_control_system}
 from a point ${\bf p}\in M$} is a subset of $M$ defined by
\[
\mathcal{R}_{{\bf p}}=\left\{\Gamma ({\bf p}, t, {\bf u}):  t\ge 0, ~~ {\bf  u}\in\mathcal U \right\}.
\]

The control system \eqref{affine_control_system} is said to be \emph{controllable from ${\bf p}\in M$} if
\begin{equation}\label{defn:controllable_from_a_pt}
\mathcal{R}_{{\bf p}}=M
\end{equation}and it is called \emph{controllable} if \eqref{defn:controllable_from_a_pt} holds for every ${\bf p}\in M$.
The \emph{orbit $\mathcal{O}_{{\bf p}}$ of the control system \eqref{affine_control_system} through a point ${\bf p}\in M$} 
is  the set of points ${\bf q}\in M$
such that either ${\bf q}$ is reachable from ${\bf p}$ or ${\bf p}$ is reachable from ${\bf q}$. \end{definition}
A basic theorem is the following

\begin{theorem} [Nagano-Sussmann theorem,  \cite{AS}] \label{Sussman thm}  $\mathcal{O}_{\bf p}$ is a connected immersed submanifold of $M$.
\end{theorem}
See \cite{AS} for the proof of Theorem \ref{Sussman thm}.   Now let $\mathcal G$    be   the Lie algebra of vector fields generated by $\vec f_1, \ldots, \vec f_m.$   
Consider  the vector space $\mathcal G({\bf p}) \subset T_{\bf p}M, $   which  is the linear span at ${\bf p}\in M$  of   the left iterated Lie brackets 
\[
[\vec f_{i_{1}},[\vec f_{i_{2}},\cdots,[\vec f_{i_{k-1}}, \vec f_{i_{k}}]\cdots]]
\]of the  vector fields $\vec f_1, \ldots, \vec f_m.$   Then we have 
\begin{theorem}[Rashevsky-Chow theorem, \cite{AS}]\label{Chow}
Let $M$ and $\mathcal{G}$ be as above. If $M$ is connected and ${\mathcal G }({\bf p}) =T_{p}M$  at every point  ${\bf p}\in M$, then $\mathcal{O}_{{\bf p}} = M, $  that is (\ref{affine_control_system}) is controllable.
\end{theorem}
\vskip 1pc

\subsection{Generalized first integrals}
Let $\Cal O$ be a germ of a smooth ($C^{\infty})$ manifold of dimension $n$, or a small open ball of $\Bbb R^n$ centered at the origin. 
A  set of real-valued smooth functions $ r_1, \cdots, r_d$ of $\Cal O$  is said to be  { \it non-degenerate }   if 
$$dr_1\we\cdots\we dr_d \neq 0.$$  The common zeros of a non-degenerate set of $d$ real-valued functions form a submanifold of codimension $d$.
\begin{lemma} \label{lemma} A smooth real-valued function $\r$ vanishes on the common zero set of a non-degenerate system  $r_1, \cdots, r_d$
of  smooth real-valued  functions  if and only if  \begin{equation}\label{ideal} \r = \sum_{k=1}^d \a_k r_k, \end{equation}
 for some smooth functions $\a_k$s.
 \end{lemma}

\begin{proof} The sufficiency of the condition is obvious.   To show the necessity of the condition (\ref{ideal}) notice that the non-degeneracy implies that $r_1, \cdots, r_d$ can be part of the coordinate system of $\Cal O.$ Thus we may assume that
$r_j=x_j$ for $j=1, \cdots, d.$   We have
$$\begin{aligned}&~~\r (x_1,\cdots, x_d, x_{d+1}, \cdots, x_n) - \r (0,\cdots,0, x_{d+1}, \cdots, x_n) \\
&= [\r (sx_1,\cdots,sx_d, x_{d+1}, \cdots, x_n)]^1_0\\ &= \int^1_0 \frac{d}{ds} \r (sx_1,\cdots,sx_d, x_{d+1}, \cdots, x_n) ds\\
&= \sum^d_{k=1}  x_k \int^1_0 \frac{\pa \r}{\pa x_k} (sx_1,\cdots,sx_d, x_{d+1}, \cdots, x_n) ds.
\end{aligned}
$$   Setting the integral of the last line $\a_k$ and recalling $x_k=r_k, $ for $k=1, \cdots, d,$  we obtain (\ref{ideal}).
\end{proof}

Lemma \ref{lemma}  is a local version in the smooth category of  Hilbert's  {\it Nullstellensatz } and the condition  (\ref{ideal}) is 
that $\r$ is in the ideal $((r^1, \cdots, r^d)).$

\begin{definition}\label{first integral defn} Let $ \vec f_1, \cdots, \vec{f}_m $   be smooth vector fields on $\Cal O$ that are linearly independent at every point of $\Cal O.$  A real-valued smooth function $\r$ is {\it  a first integral }  if  
\begin{equation} \label{first integral eq}\vec f_{\ell} \rho = 0, ~~~ \text{for all } \ell =1, \cdots, m.
\end{equation}
\end{definition}
For $m$ vector fields  on $n$-manifold there can be at most $(n-m)$ non-degenerate first integrals. This is the case that   the Frobenius integrability condition holds (see \cite{BCGGG}).

\begin{definition}\label{generalized def}  A real-valued smooth function $\r$ is {\it a generalized first integral} of 
$\vec f_1, \cdots, \vec f_{m}$  if 
\begin{equation}\label{generalized eq} {\vec f}_{\ell} \r = 0,  ~~\text{ on } \{\r = 0\},  ~~\text{ for all } \ell=1, \cdots, m. \end{equation}
\end{definition}

 The notion of the generalized first integral was first defined in \cite{AH}.    We have 
\begin{proposition}\label{observation}
Suppose that $\r$ is a smooth real-valued function with $d\r \neq 0 $ on $\Cal O$ and  $X\subset \Cal O$ is the zero set of $\r.$ Then the following are equivalent:

\noindent i) $\r$ is a generalized first integral of the vector fields $\vec f_1, \cdots, \vec f_m.$ 

\noindent ii) For each $\ell =1, \cdots, m, $ $\vec{f}_{\ell} \r $ is divisible by $\r,$  that is,  there is a smooth function $\a({\bf x})$ so that 
\begin{equation}\label{divisible} \vec{f}_{\ell} \r =\a \r. \end{equation}

\noindent iii) For each $\ell=1, \cdots, m ,$     $\vec{f}_{\ell} $ is tangent to $X.$

\noindent iv) $X$ is invariant under the flow of $\pm \vec{f}_{\ell}, \ell=1, \cdots, m.$
\end{proposition}
\begin{proof}  $i) \Longrightarrow ii)$ follows from Definition \ref{generalized eq} and Lemma \ref{lemma}.  The other implications are easy to prove or  obvious. \end{proof}  

  Using the observations of  Proposition \ref{observation}   {\it the partial integrability} of 
  systems of vector fields, or equivalently,  {\it the partial  holonomy } of  (\ref{affine_control_system}) has been studied in \cite{H1, H2, H3}
 and  their applications  to invariant submanifolds in  \cite{HK2, HP1, HP2}.

\section{Proof of the theorems }

\subsection { Proof of Theorem \ref{thm1}}    
We shall construct  $\vec f_{\ell}$, $\ell = 1, \cdots, m, $ for each of the following three cases:

\noindent 1)   {\bf Case  $m=n$.} 
Let $\vec f_{\ell}$ be the $\ell$-th column of the following $m \times  m  $  matrix:

\begin{equation} \label{matrix0} \left[
\begin{matrix}
&1  & 1+\r_1 & 1+\r_2 & 1+\r_3  &\cdots & & 1+\r_{m-1} \\
&1  & 1        &1+\r_2  & 1+\r_3  & \cdots &  & 1+\r_{m-1}\\
&1 &  1       &1          &1+\r_3  & \cdots &  & 1+\r_{m-1}\\
&1 &  1       &1         &1   & \cdots &  & 1+\r_{m-1}\\
&\vdots & \vdots & \vdots& \vdots& \cdots& & \vdots\\
&1&1&1&1&\cdots&&1+\r_{m-1} \\
&1&1&1&1&\cdots&&1+\r_{m-1} \\
 &1&1&1&1&\cdots&&1 
\end{matrix}\right], 
\end{equation} 

namely, 
\begin{equation} \label{vec def}\begin{aligned} \vec f_1 &= \sum_{j=1}^n  \pa/\pa x_j,\\
\vec f_{\ell} & = \sum_{j=1}^{\ell -1} (1+\r_{\ell -1}) \pa /\pa x_{j} + \sum_{j=\ell }^n \pa/\pa x_j , \quad \ell = 2, \cdots, m.
\end{aligned}\end{equation}   Then  (\ref{rhoj}) implies that 
\begin{equation*}\label{frho} \begin{aligned} \vec f_1 \rho_j   &= 0,  \quad  j=1,\cdots, n-1,\\
\vec f_{\ell} \rho_{\ell - 1} &= - \rho_{\ell - 1}, \quad \ell = 2, \cdots, m, \\
\vec f_{\ell}\rho_j &= 0, ~~\text{ for  the other pairs  }  (\ell, j).\end{aligned} \end{equation*}
Therefore, by Proposition \ref{observation}, $\vec f_{\ell}$ are tangent to $\{\rho_j = 0\}$, for each pair $(\ell, j).$
This implies that 
the boundary of $\Bbb R^n_{\e} $ is invariant, and therefore  the interior of $\Bbb R^n_{\e}$ is invariant,  under the affine control  system (\ref{affine}).   
   Furthermore,  the square matrix (\ref{matrix0}) is invertible on 
$\Bbb R^n_{\e} = \bigcap_{j=1}^{n-1} \{\r_j>0\}.$ Hence, by Theorem \ref{Chow}   the control-affine  system (\ref{affine}) is controllable.
\vskip 1pc

\noindent2)  {\bf Case  $m<n$ and $2m = n+1$:}  For $\ell = 1, \cdots, m, $  let  $\vec f_{\ell}$ be the $\ell$-th column of the following matrix:

\begin{equation} \label{matrix1}
\left[\begin{matrix}
&1  & 1+\r_1 & 1+\r_2 & 1+\r_3  &\cdots & & 1+\r_{m-1} \\
&1  & 1        &1+\r_2  & 1+\r_3  & \cdots &  & 1+\r_{m-1}\\
&1 &  1       &1          &1+\r_3  & \cdots &  & 1+\r_{m-1}\\
&1 &  1       &1         &1   & \cdots &  & 1+\r_{m-1}\\

&\vdots & \vdots & \vdots& \vdots& \cdots& & \vdots\\
&1&1&1&1&\cdots&&1+\r_{m-1} \\
&1&1&1&1&\cdots&&1+\r_{m-1} \\
 &1&1&1&1&\cdots&&1 \\
&&&&&&&&\\
&&&&&&&&\\
&1 &1+x_m\r_m&1&1&\cdots&&1 \\
&1&1+x_m\r_m&1+x_{m+1}\r_{m+1}&1&\cdots&&1 \\
 &1 &1+x_m\r_m&1+x_{m+1}\r_{m+1}&1+x_{m+2}\r_{m+2}&\cdots&&1 \\
 &\vdots & \vdots  & \vdots& \vdots& & & \vdots\\
&1 &1+x_m\r_m&1+x_{m+1}\r_{m+1}&1+x_{m+2}\r_{m+2}&\cdots&&1+x_{n-1}\r_{n-1} \\
\end{matrix}
\right]
\end{equation} 
thus  $\vec f_{\ell}, $  $\ell=1, \cdots, m,$   is an $n$-dimensional vector given  by 

\begin{equation}\label{2}
\begin{aligned} \vec f_1 &=  \sum_{j=1}^n  \pa/\pa x_j, \\
\vec{f}_{\ell} &= \sum_{j=1}^{\ell -1} (1+\r_{\ell -1})\frac{\pa}{\pa x_j} + \sum_{j=\ell}^{m+\ell -2} \frac{\pa}{\pa x_j} +
 \sum_{j=m+\ell -1 }^n (1+x_{m+\ell -2}  \r_{m+\ell -2}) \frac{\pa}{\pa x_j},    
 \text{ for } \ell = 2, \cdots, m. \end{aligned}\end{equation}  

Then by  easy computations  we have   
\begin{equation}\label{computation1} \begin{aligned} \vec f_1 \rho_j  &= 0,  \text{ for all } j=1, \cdots, n-1,  \\
\vec f_{\ell} \rho_{\ell -1}  &=  - \rho_{\ell - 1} ,   ~~\text{ for all }   \ell=2,\cdots, m, \\
\vec f_{\ell} \rho_{m+\ell -2} &= x_{m+\ell-2}\rho_{m+\ell-2}, ~~\text{ for all } \ell=2, \cdots, m, \\
\vec f_{\ell} \rho_j &= 0, ~~ \text{ for all other pairs } (\ell, j). 
\end{aligned}\end{equation}
Then by Proposition \ref{observation} and (\ref{computation1}) the vector fields $\vec f_1, \cdots, \vec f_m$ are tangent to the boundary of $\Bbb R^n_{\e},$    which implies that  the interior of $\Bbb R^n_{\e}$ is invariant,  under  the control-affine system (\ref{affine}).
 On the other hand

\begin{equation}\label{computation2}
[\vec f_1, \vec f_{\ell} ] = \r_{m+\ell -2}  \sum_{j=m+\ell-1}^n   \frac{\pa}{\pa x_j} ,  \text{ for  all } \ell = 2, \cdots, m,\end{equation}

so that $[\vec f_1, \vec f_{\ell}]$ is the $(\ell-1)$th column of the following matrix:

\begin{equation} \label{matrix2}
\left[
\begin{matrix}
&0  & 0 & 0 & 0  &\cdots &  0 \\
&0 &  0       &0         &0  & \cdots   &0\\
&\vdots & \vdots & \vdots& \vdots& \cdots&  \vdots\\
 &0&0&0&0&\cdots&0 \\
&&&&&&&&\\
&&&&&&&&\\
&\r_m &0&0 &0&\cdots&0 \\
&\r_m&\r_{m+1}&0&0&\cdots&0 \\
 &\r_m &\r_{m+1}&\r_{m+2}&0&\cdots&0 \\
 &\vdots & \vdots  & \vdots& \vdots&  & \vdots\\
&\r_m &\r_{m+1}&\r_{m+2}&\r_{m+3}&\cdots&\r_{n-1} 
\end{matrix}
\right].
\end{equation}

Notice that the bottom part of (\ref{matrix2}) is a $(m-1)\times (m-1)$ square matrix, which is invertible on $\Bbb R^n_{\e}$.  
Thus $m$ columns of the matrix   (\ref{matrix0}) and $(m-1)$ columns of  (\ref{matrix2}),  so that  their totality $m+(m-1) = n $ vectors are independent on 
$\Bbb R^n_{\e}$.  Now conclusion follows from Theorem \ref{Chow}.

\noindent 3) {\bf Case $2m > n+1$:  }  Proof is same as the previous case, except for that  we  need  only the columns of (\ref{matrix2})   up to $(n+1-m)$-th   in order to span the whole tangent space  of  $\Bbb R^n_{\e}$.

\subsection{Proof of Theorem \ref{thm2}}

Let $\r_j $ be the function on $\Bbb R^n$ as in (\ref{rjcircle}).  We define a function  $\sig(s)$  in a real variable $s$  as in Figure 2 by
\begin{equation*}  \sigma(s):= \sin (s-\p/2) + 1 .
\end{equation*}

\begin{figure}[h]
\includegraphics[scale=0.5]{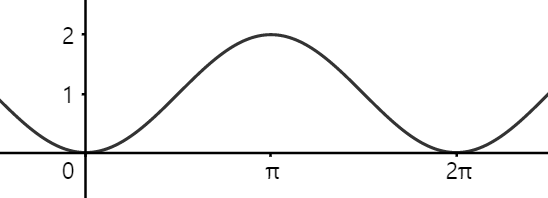}

Figure 2 $\;\;$ The function $\sig(s)$
\end{figure}

\vskip 0.0pc

Now  we define  
\begin{equation}\label{sigmaj} \sig_j ({\bf x}) := \sig(\r_j({\bf x})),  ~~ j=1, \cdots, n.\end{equation}
Immediately we have  the following
\begin{proposition} \label{sigmaj} For each $j=1, \cdots, n, $   $\sig_j ({\bf x})$   is  a smooth function of $\Bbb R^n$  with the following properties:  

\noindent i) $\sig_j$ is periodic in each variable with period $2\pi$.

\noindent ii) For  $j=1, \cdots, n-1,$  $\sig_j $ depends only on $x_{j+1} $ and $x_j$,   and $\sig_n$ depends only on $x_1$ and $x_n$.

\noindent iii) $\sig_j({\bf x}) >0, $ for all ${\bf x}\in \widetilde{\Bbb R}^n_{\e}.$

\noindent iv)  $\sig_j ({\bf x}) = 0, $   for all ${\bf x} $ with $\r_j({\bf x}) = 0, $   that is, 
$ \sig_j  \in ((\r_j)). $    \end{proposition}

Now we prove Theorem \ref{thm2} for the following three cases:
\vskip 1pc
\noindent 1) case $m=n; $  the largest choice of $m$. Let $\vec f_j$ be the $j$-th column of the following matrix:

\begin{equation}\label{matrix}
\left[
\begin{matrix}
&1  & 1+\sig_1\sig_n  & 1+\sig_2\sig_n  & \cdots &1+\sig_{n-2}\sig_n  & 1+\sig_{n-1} \sig_n \\
&1  & 1        &1+\sig_2 \sig_n & \cdots & 1+\sig_{n-2}\sig_n & 1+\sig_{n-1}\sig_n\\
&1 &  1       &1         & \cdots  & 1+\sig_{n-2}\sig_n &  1+\sig_{n-1}\sig_n\\
&\vdots & \vdots  & \vdots& \cdots& \vdots& \vdots\\
&1&1&1&\cdots&1+\sig_{n-2}\sig_n&1+\sig_{n-1}\sig_n \\
&1&1&1&\cdots&1&1+\sig_{n-1}\sig_n \\
 &1&1&1&\cdots&1&1 \\
 \end{matrix}
 \right].
\end{equation}

By easy computation we see that 

\begin{equation}\label{compute5} \begin{aligned}
&\vec f_1 \rho_j &= 0,  ~~\text{ for all } j=1, \cdots, n\quad\quad \\
&\vec f_{\ell}  \r_{\ell - 1} &= - \sig_{\ell -1} \sig_n, ~~\text{ for } \ell = 2, \cdots, m\\
&\vec f_{\ell} \r_n &= \sig_{\ell -1} \sig_n,\quad\text{ for } \ell=2, \cdots, m\\
&\vec f_{\ell}\r_j &= 0, ~~\text{ for all  other pairs } (\ell, j).
\end{aligned}
\end{equation}

From the second equation of (\ref{compute5}) and Proposition \ref{sigmaj}-({\it iv})  we have 
\begin{equation}\label{rhol} \begin{aligned} \vec f_{\ell} \r_{\ell -1} &\in ((\r_{\ell -1}))\\
\vec f_{\ell} \r_n &\in ((\r_n)).\end{aligned}.\end{equation}
It follows from   (\ref{compute5})-(\ref{rhol})  that the  vector fields $\vec f_1, \cdots, \vec f_m$ are tangent to  the zero sets $\{\r_j=0\} $,  which implies that the boundary of $\widetilde{\Bbb R}^n_{\e} $ is invariant,  and also,  the interior of $\widetilde R^n_{\e}$  is   invariant   under the  control-affine system (\ref{affine}).   Furthermore, by  Proposition \ref{sigmaj}-({\it iii}) the matrix (\ref{matrix}) is invertible on $\widetilde{\Bbb R}^n_{\e}$,  which implies that $\vec f_{\ell}, \ell=1, \cdots, m, $ span the whole tangent space of $\widetilde{\Bbb R}^n_{\e}$.   
\vskip 1pc

\noindent 2) Case  that $n$ is odd,  $m$ is odd,  and $2m=n+1.$

Let us consider the following $n\times m $ matrix, whose top part is the same as the first $m$ columns and the first $m$ rows of  (\ref{matrix}) and the bottom part has $(n-m) $ rows.
For $\ell =1, \cdots, m$,  let $\vec f_{\ell}$ be the $\ell$-th column  of the following matrix:

\begin{equation} \label{matrix5}
\left[
\begin{matrix}
1  & 1+\sig_1\sig_n  & 1+\sig_2\sig_n & 1+\sig_3\sig_n & \cdots &1+\sig_{m-2}\sig_n  & 1+\sig_{m-1} \sig_n \\
1  & 1        &1+\sig_2 \sig_n& 1+\sig_3\sig_n & \cdots & 1+\sig_{m-2}\sig_n & 1+\sig_{m-1}\sig_n\\
1 &  1       &1        & 1+\sig_3\sig_n & \cdots  & 1+\sig_{m-2}\sig_n &  1+\sig_{m-1}\sig_n\\
1 &  1       &1        & 1 & \cdots  & 1+\sig_{m-2}\sig_n &  1+\sig_{m-1}\sig_n\\

\vdots & \vdots  & \vdots&\vdots& \cdots& \vdots& \vdots\\
1&1&1&1&\cdots&1+\sig_{m-2}\sig_n&1+\sig_{m-1}\sig_n \\
1&1&1&1&\cdots&1&1+\sig_{m-1}\sig_n \\
1&1&1&1&\cdots&1&1 \\
&&&&&&\\
&&&&&&\\
1&1+A_1&1- B_1&1&\cdots &1&1\\ 
1&1+B_1&1+A_1&1&\cdots&1&1\\
1&1&1&1+A_2&\cdots&1&1\\
1&1&1&1+B_2&\cdots&1&1\\
\vdots&\vdots&\vdots&\vdots&&\vdots&\vdots\\
1&1&1&1  &\cdots &1+A_{(m-1)/2}&1 +B_{(m-1)/2}\\
1&1&1&1 &\cdots  &1-B_{(m-1)/2}&1+A_{(m-1)/2}\\
\end{matrix}\right].
\end{equation} 
The bottom part of  (\ref{matrix5}) has $2\times 2$ square submatrices
\begin{equation}\label{submatrix}\left[\begin{matrix} 
1+A_{\l} &1-B_{\l}\\ 1+B_{\l}&1+A_{\l} 
\end{matrix}\right],   ~~~\quad  \l=1, 2, \cdots, (m-1)/2,\end{equation}
where 

\begin{equation}\label{sincos}\begin{aligned}
A_{\l} &= \sin x_{m+2(\l -1)}\sig_{m+2(\l -1)}\sig_{m+2(\l -1)+1}\\
B_{\l} &= \cos x_{m+2(\l -1)}\sig_{m+2(\l -1)}\sig_{m+2(\l -1)+1},\end{aligned}\end{equation}
 so that  the last block with $\l =(m-1)/2$ is $$ \left[\begin{matrix} 1+A_{(m-1)/2} & 1 - B_{(m-1)/2}\\1+B_{(m-1)/2}&1+A_{(m-1)/2}\end{matrix}\right], $$
where    $$\begin{aligned} A_{(m-1)/2} &= \sin x_{n-2} \sig_{n-2}\sig_{n-1}\\ B_{(m-1)/2} &= \cos x_{n-2} \sig_{n-2}\sig_{n-1}.\end{aligned}$$
Since $\vec f_1 = \sum_{i=1}^n \pa/\pa x_i $ calculation shows that for each $ j=1, \cdots, n, $ and each $\l=1, \cdots, (m-1)/2 $ we have
\begin{equation}\label{calculation}\begin{aligned} \vec f_1 \sig_j &= \sig'(\r_j) \vec f_1 \r_j = 0, \\
\vec f_1 A_{\l} &= \cos x_{m+2(\l-1)}\sig_{m+2(\l-1)}\sig_{m+2(\l-1)+1}\\
\vec f_1 B_{\l} &= -\sin x_{m+2(\l-1)}\sig_{m+2(\l-1)}\sig_{m+2(\l-1)+1}. \\
\end{aligned}\end{equation}
Therefore, expressing the brackets $[\vec f_1, \vec f_{\ell}] $ for $\ell=2,3,4,5 $   as  column vectors, we have

\begin{equation}\label{matrix6}\left[
\begin{matrix}
0 & 0&0&0\\ \vdots&\vdots&\vdots &\vdots \\ 0&0&0&0\\ &&&  \\ \cos x_m \sig_m\sig_{m+1} &  \sin x_m \sig_m\sig_{m+1}&0&0\\
-\sin x_m\sig_m\sig_{m+1} &\cos x_m\sig_m\sig_{m+1}&0&0\\ 0&0 &\cos x_{m+2}\sig_{m+2}\sig_{m+3}  & \sin x_{m+2}\sig_{m+2}\sig_{m+3}\\ 0&0&-\sin x_{m+2}\sig_{m+2}\sig_{m+3} &\cos x_{m+2}\sig_{m+2}\sig_{m+3}\\
 \vdots&\vdots &\vdots &\vdots\\0&0&0&0\\ 0&0&0&0 \\ 
\end{matrix}\right]
\end{equation}  
 
Notice that $\sig_j$'s are positive valued on $\widetilde{\Bbb R}^n_{\e} = 
\bigcap_{i=1}^n \{\r_i>0\}.$    Thus  for each $\ell = 2, \cdots, m, $  $[\vec f_1, \vec f_{\ell}]  $  divided by a positive scalar, is the $(\ell - 1)$-th column of the following matrix;

\begin{equation}\label{matrixafter}
\left[
\begin{matrix}
0 & 0&0&0 &\cdots &0&0\\ 
\vdots&\vdots&\vdots &\vdots &&\vdots&\vdots \\
 0&0&0&0&\cdots &0&0\\ 
&&& && \\
 \cos x_m & \sin x_m  &0 &0 &\cdots&0&0\\
-\sin x_m &\cos x_m &0&0 &\cdots &0&0\\ 0&0 &\cos x_{m+2}  &\sin x_{m+2}&\cdots &0&0\\
 0&0&-\sin x_{m+2} &\cos x_{m+2}\\
 \vdots&\vdots &\vdots &\vdots &&\vdots&\vdots\\
0&0&0&0&\cdots&\cos x_{n-3} & \sin x_{n-3}\\ 
0&0&0&0&\cdots &-\sin x_{n-3} &\cos x_{n-3}\\ 
\end{matrix}
\right].
\end{equation}
The above $(m-1)$ vectors are linearly independent and we see that $\vec f_{\ell}, $ $\ell=1, \cdots, m $    and $[\vec f_1, \vec f_{\ell}], $ ${\ell}=2, \cdots, m$,  span the tangent space of $\widetilde{\Bbb R}^n_{\e}$.  

\vskip 1pc

\noindent 3) Case that  $n$ is even, $m$ is even, and  $2m\ge n+2:$    It is enough to construct for the case $2m=n+2.$ 
Since $m$ is even and we want $\vec f_1$ to be a constant vector field $\sum_{j=1}^n \pa/\pa x_j$ as in the previous case.  Since $m-1$ is odd, 
we can construct $(m-2)/2$ blocks of size $2\times 2$ in the same way as in the previous case. Construction is same except for the brackets
$[\vec f_1, \vec f_{\ell}]$ ends up with $\ell=m-1.$

\section{ The authors' motivation  and some future problems}

 The  present paper is an outcome of the authors'  attempt to find applications to control theory of the generalizations of the Frobenius theorem on involutivity \cite{H1, H2, H3}.  
The quasi-linear systems of first order partial differential equations that arise from various problems in geometry and physics are often the problem of deciding the partial integrability of the associated Pfaffian systems.  In particular, 
 the notion of the generalized first integral and its construction are useful for finding the partition of the configuration space into invariant submanifolds of dynamical systems, see \cite{HK2, HP1, HP2, HP3}.   The collision avoidance problems belong to this category.  It turns out that even $1$-dimensional cases are not obvious, for which we 
constructed vector fields $\vec f_{\ell}$, so that  the  associated obstacle and its complement in the product space are  invariant  of the dynamics of 
 the control-affine system (\ref{first affine}).   In general dimensions, however, the methods of the exterior differential system are more convenient 
for finding the generalized first integrals.
 We shall explain the ideas briefly now.

\subsection{Exterior differential system approach to the first integrals and the  generalized first integrals}

Let 
\begin{equation}\label{vfs} \vec f_1, \cdots, \vec f_m \end{equation}   
be a system of  linearly independent smooth vector fields defined on a small open set   $\Cal O$  of an $n$-dimensional smooth manifold.  Let 
\begin{equation}\label{Pfaffian} \th := (\th^1, \cdots, \th^s), \quad s:=n-m,  \end{equation}  
be a system of independent $1$-forms that annihilate $\vec f_j, $   $j=1, \cdots, m, $
namely,  
\begin{equation*} \th^{\a} (\vec f_{\ell})= 0, \quad \text{for all } \a=1, \cdots, s,  ~~\forall \ell  = 1, \cdots, m. \end{equation*}  We shall call (\ref{Pfaffian}) 
the  {\it{Pfaffian system }} associated with (\ref{vfs}).  
 Let  
$$\O^*(\Cal O):= \bigoplus_{k=0}^n \O^k(\Cal O)$$ be  the exterior algebra of smooth differential forms of $\Cal O$, 
where
  $\O^k(\Cal O)$ is the module of differential $k$-forms for  $k=1, \cdots, n, $  and $ \O^0(\O):= C^{\infty}(\O) $  
is  the ring of smooth real-valued functions of $\Cal O.$   We denote by   $(( \cdots ))$ an  algebraic ideal  of $\Cal O^*(\Cal O)$ 
generated by what are inside  the double parenthesis $(( ~~~ )).$ 
 Then  Definition \ref{first integral defn} and Definition \ref{generalized def} are equivalent to the following
\begin{definition}\label{2nd defn first integral} 
A smooth function $\r$ is a first integral of (\ref{vfs})   if
\begin{equation}\label{first integral def2}  d\r \in ((\th)), \quad \end{equation}  where $ ((\th)) = ((\th^1, \cdots, \th^s)). $
 A smooth function $\r$ is a generalized first integral of (\ref{vfs}) if 
\begin{equation}\label{generalized def2} d\r \in ((\r, \th)).\end{equation}
\end{definition}

\begin{definition} The system of $2$-forms 
$$d\th,  \text{ modulo } ((\th))$$ is called the {\it torsion} of (\ref{Pfaffian}).\end{definition}

Torsion is the obstruction to the integrability of (\ref{Pfaffian}) and by analyzing its exterior-algebraic properties we can determine various partial integrabilities.  Among them we can find  the generalize first integrals explicitly, as we shall now explain briefly  for  the cases $n=3$ and
 $m=2: $

\vskip 1pc
Suppose that $\vec f_1$ and $ \vec f_2$  are  smooth vector fields that are linearly independent on a small connected open subset $\Cal O \subset \Bbb R^3 = \{(x,y,z)\}.$  Let $\th$ be a nowhere vanishing $1$-form that annihilates $\vec f_{\ell}, \ell=1,2. $    Choose $1$-forms $\o^1, \o^2$ so that $(\th, \o^1, \o^2)$ is a coframe of $\Cal O.$
Let $T$ be the coefficient of  the torsion: 
\begin{equation}\label{40} d\th \equiv T \o^1\we \o^2, \quad\text{mod } ((\th)).\end{equation}
By (\ref{generalized def2})   
\begin{equation}\label{41} d\r = \r \psi + \l \th,  \quad \l\neq 0, \end{equation}
for some $1$-form $\psi$ and for some function $\l.$  Now apply $d$ to (\ref{41}) and {\it mod out } by $((\th)) $ and then substitute (\ref{41}) for $d\r$:
\begin{equation}\label{42} \begin{aligned} 0 &= d\r \we \psi + \r d\psi + d\l\we \th + \l d\th\\
&\equiv  \l T \o^1\we \o^2, \quad\text{mod } ((\r, \th)). \end{aligned}\end{equation} 

Since $\l\neq 0$ and $\th, \o^1,\o^2$ are independent, (\ref{42}) implies 
\begin{equation}\label{43} T\in ((\r)),\end{equation}   which means that $T$ is divisible by $\r$.  
Thus for a smooth function $\r$ to be a generalized first integral of $\th$, a necessary condition is that $\r$ is a smooth non-degenerate factor of $T$. 
Now we pass $\r$ through the test: if  (\ref{generalized def2})  holds  this $\r$ is indeed a generalized first integral.

\vskip 1pc
\begin{example} Let $\Cal O = \Bbb R^3 =\{(x,y,z)\}$ 
and 
\begin{equation*} \begin{aligned} \vec f_1 &= \frac{\pa}{\pa x}\\ \vec f_2 &= \frac{\pa}{\pa y} + xyz\frac{\pa}{\pa z}.\end{aligned}\end{equation*} 
Then \begin{equation}\label{example theta} \th = dz - xyz dy 
\end{equation}  annihlates $\vec f_{\ell}, $  $\ell=1,2.$       In terms of coframe $(\th, dx, dy)$ we have
\begin{equation*} d\th \equiv -yz ~dx\we dy, \quad\text{mod } ((\th)), \end{equation*} thus $T(x,y,z) := -yz$   and 
a single $2$-form 
\begin{equation*} T dx\we dy 
\end{equation*}
is the torsion of $\th.$  Now we check whether (\ref{generalized def}) holds for each of the factors of $T:$

\begin{equation*}\label{ideal check} \begin{aligned} dy &\notin (( y, dz-xyz ~dy)) \\  dz &\in (( z, dz-xyz~ dy)), \end{aligned}\end{equation*}
$z$ is a generalized first integral and $y$ is not.
\end{example}

 In this paper we went through a reverse  process: For a prescribed function 
 $\r$  we constructed   a Pfaffian system $\th$, or equivalently,  vector fields $\vec f_{\ell}, \ell = 1, 2,\cdots  $  for which  $\r$  is a generalized first integral. We refer to \cite{G}  for the partial integrability as a  generalization of the Frobenius theorem,   to 
\cite{LR, PLTS, PR, T} for the applications of exterior differential systems to control theory, and to
 \cite{BCGGG, CDF,   J97} for general references.

\subsection{ Number of vector fields} 

We constructed $m$ vector fields $\vec f_{\ell}$, $\ell = 1, \cdots, m, $   so that they span, together with their brackets of first order, the whole tangent space of dimension $n$.  However,  the controllability can be achieved with a much smaller  number $m \in O(n^{1/2})$ of 
  vector fields  and their bracket of first order.  If the Lie brackets up to order $k$ are to be used to span the whole tangent spaces of the configuration space, the number of vector fields can be as small as  $m\in O(n^{1/2^k}).$

\subsection{Optimality}  We did not discuss the optimality in this paper.  In dimension $1,$   minimizing the total distance traveled is not a problem. 
In higher dimensions, however,   minimizing distance, or more generally,  minimizing a certain action is at the heart of  path planning.
Results on optimal control of collision avoidance and obstacle avoidance are found  in 
\cite{   BCC, HS, MWCD, TF}.    When  the obstacle and the  configuration space off the obstacle  have symmetries  
then the variational symmetry gives rise to conservation laws due to E. Noether's theorem, see \cite{FT, Tor}.

\subsection{Monotonicity}  When the configuration space is endowed with a partial order, a dynamical system $\dot{\bf x} = \vec F({\bf x}, {\bf u})$, with input ${\bf u}$, is said to be monotone if  $\G({\bf x}, t, {\bf u})$ preserves the order for all $t\ge 0$, where $\G$ is the flow map as defined in \S 2.
We refer  to  \cite{ASontag, HS} for monotone dynamical systems,  and to 
 \cite{Ch, HS2, MWF,  Smith, Sontag} for monotone control systems.  The controllability and reachability for monotone control 
systems and the  relevant results are found in   \cite{CF, Coo, MGW,  Val}.
Observe that the monotonicity or the order structure plays no role in the present paper.     The circle  $S^1$ has no natural order  and  hence the corresponding dynamics on $T^n $  have no natural notion of monotonicity.  However,  $\Bbb R^1$ has a natural order and hence our configuration space $\Bbb R^n_{\e} $ has partial order defined by the orthant $\Bbb R^n_{+}$.   
  Our input functions ${\bf u} = (u^1, \cdots, u^m)$ are allowed to take  negative values and        (\ref{affine})
is not necessarily monotone in general.  It seems to the authors that 
for the particular affine control systems  in Theorem 1.1,  we can determine the monotonicity by using the criteria 
given by   Corollary III.3 of \cite{ASontag},  to conclude  that
\vskip 1pc
 \noindent i) The case $m<n$ and $2m=n+1$  given by (\ref{matrix1}) is not monotone.  

\noindent  ii) The  case $m=n$, if the input functions ${\bf u}$  satisfy 
$$ u_2 \ge u_3 \ge \cdots \ge u_m \ge 0,$$
the dynamical system  (\ref{matrix0}) is cooperative.

\subsection{Some further problems}  
\vskip 1pc
\noindent 1) Collision avoidance in $M^2$  of dimension $2, $ namely,   $\Bbb R^2, S^2, T^2$, and so forth: Construct a control-affine system (\ref{first affine}), so as to avoid collision and be controllable. 
\vskip 1pc
\noindent 2) Construct a control-affine system (\ref{first affine}) on  $M^2$ that avoids prescribed obstacles.  If the configuration space has symmetry, for example, the annulus 
bounded by  two concentric circles,  find the necessary conditions for optimality that come from the variational symmetry and  Noether's theorem.  Can one give a complete metric on the annulus so that the geodesics are  optimal  in any practical sense?
\vskip 1pc
\noindent 3) Systems of   Kolmogorov type equations for three species:    Let $x_j,$ $j=1,2,3,$  be the population of three  species.   We consider a control system $ \dot{\bf x} = \vec F({\bf x}, {\bf u}) $, with  control parameter ${\bf u}$.  For a prescribed open subset  $\Cal O$ with smooth boundary  of the configuration space,  determine the control ${\bf u}(t)$ so that  $\Cal O$ is  invariant under the flow maps.   Control problems of Kolmogorov type equations have been studied in  \cite{Chen, HPre, Kuz, LZY, ZTCT}.

\section*{Declarations}

%Some journals require declarations to be submitted in a standardised format. Please check the Instructions for Authors of the journal to which you are submitting to see if you need to complete this section. If yes, your manuscript must contain the following sections under the heading `Declarations':

\begin{itemize}
\item Funding : The first author was  supported by  NRF-Republic of Korea with  grants 0450-20210049. 
\item Conflict of interest/Competing interests : Not applicable. 
\item Ethics approval : Not applicable. 
\item Consent to participate : Not applicable. 
\item Consent for publication : Not applicable. 
\item Availability of data and materials : Not applicable. 
\item Code availability  : Not applicable. 
\item Authors' contributions : The authors contributed equally to this work. 
\end{itemize}

%\noindent
%If any of the sections are not relevant to your manuscript, please include the heading and write `Not applicable' for that section. 


\end{document}